\documentclass{article}
\usepackage[utf8]{inputenc}

\usepackage{amsmath}
\usepackage{amsfonts}
\usepackage{mathtools}
\usepackage{amsthm}
\usepackage{amssymb}

\usepackage{lineno}
\newtheorem{theorem}{Theorem}
\newtheorem{lemma}{Lemma}
\newtheorem{corollary}{Corollary}

\title{Solubility of Additive Forms of Twice Odd Degree over Ramified Quadratic Extensions of $\mathbb{Q}_2$}
\author{Drew Duncan \and David B. Leep}
\date{\today}

\begin{document}

\maketitle

\begin{abstract}
We determine the minimal number of variables $\Gamma^*(d, K)$ which guarantees a nontrivial solution for every additive form of degree $d=2m$, $m$ odd, $m \ge 3$ over the six ramified quadratic extensions of $\mathbb{Q}_2$.  We prove that if $K$ is one of $\{\mathbb{Q}_2(\sqrt{2}), \mathbb{Q}_2(\sqrt{10}), \mathbb{Q}_2(\sqrt{-2}), \mathbb{Q}_2(\sqrt{-10})\}$, $\Gamma^*(d,K) = \frac{3}{2}d$, and if $K$ is one of $\{\mathbb{Q}_2(\sqrt{-1}), \mathbb{Q}_2(\sqrt{-5})\}$,  $\Gamma^*(d,K) = d+1$.  The case $d=6$ was previously known.
\end{abstract}

\section{Introduction}

Homogeneous forms of the type
\begin{equation}
\label{eq}
a_1x_1^d + a_2x_2^d + \ldots + a_sx_s^d
\end{equation} where $a_1, \ldots, a_s$ belong to some field $K$ are known as \textit{additive forms} of degree $d$.  A famous conjecture of Artin claimed that any homogeneous form over a local field $K$ of degree $d$ in $d^2+1$ variables has a nontrivial zero regardless of the choice of coefficients from $K$.  Many well-known counterexamples to this conjecture have been discovered (see \cite{MR197450}, \cite{greenberg1969lectures}), but none of them have been additive forms.  It has therefore been proposed that the conjecture holds when restricted to additive forms.  We will refer to this as Artin's Additive Form Conjecture.

Let $\Gamma^*(d, K)$ represent the minimum number of variables $s$ such that every form (\ref{eq}) is guaranteed to have a nontrivial zero, regardless of the choice of $a_i \in K$.  In this language, Artin's Additive Form Conjecture posits that $\Gamma^*(d, K) \le d^2+1$. Davenport and Lewis \cite{davenport1963homogeneous} introduced the method of contraction, and established not only the truth of Artin's Additive Form Conjecture for every field of p-adic numbers $\mathbb{Q}_p$, but the optimality of this bound when $d=p-1$.  They noted, however, that when $d \ne p-1$ very often a much lower value of $\Gamma^*(d, \mathbb{Q}_p)$ can be found.  Further, little is known in general about $\Gamma^*$ for finite extensions of $\mathbb{Q}_p$.

In \cite{knapp2016solubility}, Knapp, using the method of contraction, showed that for every ramified quadratic extension $K$ of $\mathbb{Q}_2$, we have $\Gamma^*(6,K) \le 9$, with equality holding for $K \in \{\mathbb{Q}_2(\sqrt{2}), \mathbb{Q}_2(\sqrt{10}), \mathbb{Q}_2(\sqrt{-2}), \mathbb{Q}_2(\sqrt{-10})\}$.  For the remaining two extensions, $\mathbb{Q}_2(\sqrt{-1})$ and $\mathbb{Q}_2(\sqrt{-5})$, Knapp further showed that $\Gamma^*(6,K) \ge 7$ and conjectured that $\Gamma^*(6,K) = 7$, and the authors recently showed that Knapp's conjecture holds \cite{duncan2020solubility}.

Here, similar techniques will be used to extend these results to forms of degree $d=2m$, where $m$ is an odd integer, at least 3.

\begin{theorem} \label{theo}
Let $d=2m$, where $m$ is an odd integer at least 3.
\begin{itemize}
    \item If $K \in \{\mathbb{Q}_2(\sqrt{2}), \mathbb{Q}_2(\sqrt{10}), \mathbb{Q}_2(\sqrt{-2}), \mathbb{Q}_2(\sqrt{-10})\}$, then $\Gamma^*(d,K) = \frac{3}{2}d$.
    \item If $K \in  \{\mathbb{Q}_2(\sqrt{-1})$, $\mathbb{Q}_2(\sqrt{-5})$\}, then $\Gamma^*(d,K) = d+1$.
\end{itemize}
\end{theorem}

\section{Preliminaries}

Let $K$ denote one of the ramified quadratic extensions of $\mathbb{Q}_2$, and $\mathcal{O}$ denote its ring of integers.  Without loss of generality, assume $a_i \in \mathcal{O}\backslash\{0\}$.  Let $\pi$ be a uniformizer (generator of the unique maximal ideal) of $\mathcal{O}$, so that for any $c \in \mathcal{O}$, $c$ can be written $c = c_0 + c_1\pi + c_2\pi^2 + c_3\pi^3 + \ldots$, with $c_i \in \{0,1\}$.  The choices of uniformizer and corresponding representation of 2 are listed in Table 1 (cf. Table 1 of \cite{knapp2016solubility}).

\begin{table}[]
    \label{fields}
    \caption{Uniformizer and representation of 2}
    \centering
    \begin{tabular}{|c|c|c|}
        \hline 
         $K$ & $\pi$ & $2 \pmod{\pi^4}$ \rule{0pt}{2.6ex} \\
         \hline
         \rule{0pt}{2.6ex} \rule[-0.9ex]{0pt}{0pt}
         $\mathbb{Q}_2(\sqrt{2})$ & $\sqrt{2}$ & $\pi^2$ \rule{0pt}{2.6ex} \rule[-0.9ex]{0pt}{0pt}\\
         \hline
         $\mathbb{Q}_2(\sqrt{-2})$ & $\sqrt{-2}$ & $\pi^2$ \rule{0pt}{2.6ex} \\
         \hline
         $\mathbb{Q}_2(\sqrt{10})$ & $\sqrt{10}$ & $\pi^2$ \rule{0pt}{2.6ex} \\
         \hline
         $\mathbb{Q}_2(\sqrt{-10})$ & $\sqrt{-10}$ & $\pi^2$ \rule{0pt}{2.6ex} \\
         \hline
         $\mathbb{Q}_2(\sqrt{-1})$ & $1 + \sqrt{-1}$ & $\pi^2 + \pi^3$ \rule{0pt}{2.6ex} \\
         \hline
         $\mathbb{Q}_2(\sqrt{-5})$ & $1 + \sqrt{-5}$ & $\pi^2 + \pi^3$ \rule{0pt}{2.6ex} \\
         \hline
    \end{tabular}
\end{table}

Factoring out the highest power of $\pi$, the coefficient $c$ of any variable $x$ can be written in the form $c = \pi^r(c_0 + c_1\pi + c_2\pi^2 + c_3\pi^3 + \ldots)$, $c_0 \neq 0$.  Such a variable is said to be at \textit{level $r$}, and we will refer to the value of $c_1$ as its \textit{$\pi$-coefficient}, and more generally $c_k$ as its \textit{$\pi^k$-coefficient}.  We will also have occasion to refer to selections of coefficients as a \textit{coefficient class}.  By the change of variables $\pi^r x^d = \pi^{r-id}(\pi^i x)^d = \pi^{r-id}y^d$ for $i \in \mathbb{Z}$, we will consider the level of a variable modulo $d$ for the remainder of the paper.

Multiplying a form by $\pi$ increases the level of each variable by one, and does not affect the existence of a nontrivial zero.  Considering the levels of variables modulo $d$, applying this transformation any number of times effects a cyclic permutation of the levels.  This is useful for arranging the variables in an order which is more convenient, a process to which we will refer as \textit{normalization}. See Lemma 3 of \cite{davenport1963homogeneous} for a proof of the following Lemma.

\begin{lemma}
Given an additive form of degree $d$ in an arbitrary local field $K$, let $s$ be the total number of variables, $s_i$ be the number of variables in level $i \pmod{d}$. By a change of variables, the form may be transformed to one with:
\begin{align}
\begin{split}
s_0 \ge \frac{s}{d},  
\end{split}
\begin{split}
s_0 + s_1 \ge \frac{2s}{d},
\end{split}
\begin{split}
\ldots,
\end{split}
\begin{split}
s_0 + \ldots + s_{d-1} = s
\end{split}
\end{align}
\end{lemma}

Consider two variables $x_1,x_2$ in the same level, and without loss of generality assume their coefficients $a_1,a_2$ are not divisible by $\pi$ (by the above cyclic permutation of levels).  Suppose there is an assignment $x_1 = b_1, x_2 = b_2$ such that $\pi^k | (a_1 b_1^d + a_2 b_2^d)$.  Then the change of variables $x_1=b_1 y, x_2 = b_2 y$ yields a form in which the two variables $x_1, x_2$ are replaced with the new variable $y$ at least $k$ levels higher.  If this new form has a nontrivial zero, then there is a nontrivial zero of the original form.  This transformation is known as a \textit{contraction}, and it is key to all of the results that follow.

The following lemma establishes the existence of the types of contractions for $d=2m$, $m$ odd, at least 3, which will be used in this proof.  Each type of contraction is named for convenient reference.

\begin{lemma}~
\begin{enumerate}
    \item Two variables in the same level with differing $\pi$-coefficients can be contracted to a variable one level up having a $\pi$-coefficient of one's choosing. (\textit{d}-contraction)
    
    \item Two variables in the same level with the same $\pi$-coefficient can be contracted to a variable exactly two levels up. (\textit{s2}-contraction)
    
    \item Two variables in the same level with the same $\pi$-coefficient can be contracted to a variable at least three levels up. (\textit{s3}-contraction)
    
    \item Suppose $K \in \{\mathbb{Q}_2(\sqrt{-1}), \mathbb{Q}_2(\sqrt{-5})\}$.  Among three variables in the same level with the same $\pi$-coefficient, two can be contracted to a variable at least four levels up. (\textit{t}-contraction)
    
    \item Suppose $K \in \{\mathbb{Q}_2(\sqrt{2}), \mathbb{Q}_2(\sqrt{10}), \mathbb{Q}_2(\sqrt{-2}), \mathbb{Q}_2(\sqrt{-10})\}$.  Among three variables in the same level with the same $\pi$-coefficient, two can be contracted to a variable exactly two levels up having the same $\pi$-coefficient. (\textit{st}-contraction)
\end{enumerate}
\end{lemma}
\begin{proof}~
Using the fact that $2 \equiv \pi^2 \pmod{\pi^3}$,

\smallskip

$(1 + a\pi + b\pi^2 + c\pi^3)^{2m} \equiv (1 + 2a\pi + (2b+a^2)\pi^2 + (2c + 2ab)\pi^3)^m \equiv$

$(1 + a^2\pi^2 + a\pi^3)^m \equiv 1 + ma^2\pi^2 + ma\pi^3 \equiv$

$1 + a\pi^2 + a\pi^3 \pmod{\pi^4}.$

Therefore, the $d$\textsuperscript{th} powers modulo $\pi^4$ are $0$ and $1 + c\pi^2 + c\pi^3$ for $c \in \{0,1\}$.

\bigskip

Using $2 \equiv \pi^2 + j\pi^3 \pmod{\pi^4}, j \in \{0, 1\}$,

\smallskip

$(1 + a_1\pi + a_2\pi^2 + a_3\pi^3) + (1 + b_1\pi + b_2\pi^2 + b_3\pi^3)(1 + c\pi^2 + c\pi^3) \equiv $

$(1 + a_1\pi + a_2\pi^2 + a_3\pi^3) + [1 + b_1\pi + (b_2 + c)\pi^2 + (b_3 + b_1c + c)\pi^3] \equiv$

$2 + (a_1 + b_1)\pi + (a_2 + b_2 + c)\pi^2 + (a_3 + b_3 + b_1c + c)\pi^3 \equiv$

$(a_1 + b_1)\pi + (1 + a_2 + b_2 + c)\pi^2 + (j + a_3 + b_3 + b_1c + c)\pi^3 \pmod{\pi^4}$.

\bigskip

First suppose that the two variables have differing $\pi$-coefficients.  We have $a_1 + b_1 = 1$, and so the resulting variable is one level higher.  Choose $c$ so that $1 + a_2 + b_2 + c$ (the $\pi$-coefficient of the resulting variable) is the desired value $\pmod{2}$.  This proves (1).

\bigskip

Suppose now that $a_1 = b_1$.

\smallskip

$(a_1 + b_1)\pi + (1 + a_2 + b_2 + c)\pi^2 + (j + a_3 + b_3 + b_1c + c)\pi^3 \equiv$

$2a_1\pi + (1 + a_2 + b_2 + c)\pi^2 + (j + a_3 + b_3 + a_1c + c)\pi^3 \equiv$

$(1 + a_2 + b_2 + c)\pi^2 + (a_1 + j + a_3 + b_3 + a_1c + c)\pi^3 \pmod{\pi^4}$

\smallskip

Choosing $c$ so that $1 + a_2 + b_2 + c \equiv 1 \pmod{2}$ gives a variable raised exactly two levels.  This proves (2).  Choosing $c$ so that $1 + a_2 + b_2 + c \equiv 0 \pmod{2}$ gives a variable raised at least three levels.  This proves (3).

\bigskip

If $K \in \{\mathbb{Q}_2(\sqrt{-1}), \mathbb{Q}_2(\sqrt{-5})\}$, then $j=1$.  Choose $c$ so that $c \equiv 1 + a_2 + b_2 \pmod{2}$.

\smallskip

$(1 + a_2 + b_2 + c)\pi^2 + (a_1 + j + a_3 + b_3 + a_1c + c)\pi^3 \equiv$

$(a_3 + b_3 + (1 + a_1)(a_2 + b_2))\pi^3 \pmod{\pi^4}$

\smallskip

If $a_1 = 1$, the expression becomes $a_3 + b_3$, and by the pigeonhole principle the two variables can be chosen so that their $\pi^3$-coefficients are the same.  If $a_1=0$, the expression becomes $a_2 + b_2 + a_3 + b_3$.  If there is a pair in the same $\pi^2,\pi^3$-coefficient class, then choose that pair.  If no such pair exists, then by the pigeonhole principle there exists a pair with both differing $\pi^2$- and $\pi^3$-coefficients.  It follows that $a_2 + b_2 + a_3 + b_3 \equiv 0 \pmod{2}$.  This proves (4).

\bigskip

If $K \in \{\mathbb{Q}_2(\sqrt{2}), \mathbb{Q}_2(\sqrt{10}), \mathbb{Q}_2(\sqrt{-2}), \mathbb{Q}_2(\sqrt{-10})\}$, then $j=0$.  Choose $c$ so that $c \equiv a_2 + b_2 \pmod{2}$.

\smallskip

$(1 + a_2 + b_2 + c)\pi^2 + (a_1 + j + a_3 + b_3 + a_1c + c)\pi^3 \equiv$

$\pi^2 + (a_1 + a_3 + b_3 + (1 + a_1)(a_2 + b_2))\pi^3 \pmod{\pi^4}$

\smallskip

As above, the two variables can be chosen so that $(a_3 + b_3 + (1 + a_1)(a_2 + b_2)) \equiv 0 \pmod{2}$.  This proves (5).
\end{proof}

Because the proof of Theorem \ref{theo} involves inspecting numerous subcases, it is useful to introduce a compact notation to convey important information about particular forms and contractions performed on them.  The notation $(s_0, s_1, s_2, \ldots)$ indicates any form having \textit{at least} $s_i$ variables at level $i$.  Levels which are omitted in this notation may be assumed to contain as few as no variables.  For example, $(4, 2, 0, 1)$ indicates a form with at least four variables at level 0, at least two variables at level 1, as few as no variables at level 2, at least one variable in level 3, and as few as no variables at any higher level.

To indicate that variables in a particular level fall into certain $\pi$-coefficient classes, the number of variables in that level are partitioned into two numbers stacked vertically.  For example, $(^3_1, 2, 0, 1)$ indicates a form as above with four variables in level 0, three of which are in one $\pi$-coefficient class, and one of which is in the other.  Note that this notation does not give any indication of which $\pi$-coefficient class contains which number of variables.

We indicate that a contraction performed on a certain type of form results in a form of another specified type with an arrow, labeled with the contraction performed.  For example, consider a form with at least two variables in a level with the same $\pi$-coefficient, and having at least two variables two levels higher with differing $\pi$-coefficients.  The first two can be used to perform an \textit{s2}-contraction, resulting in a variable exactly two levels higher:

$$({}^{2}_{0}, 0, {}^{1}_{1}) \xrightarrow{s2} (0, 0, {}^{2}_{1})$$

In some places square brackets are used instead of parentheses to succinctly indicate an \textit{exact} number of variables in the specified levels and coefficient classes (e.g., $[^3_1, 2, 0, 1]$).

Finally, we give a statement of Hensel's Lemma specific to the needs of the proof below (see Theorem 2.1 of \cite{leep2018diagonal}).  This is the same version of Hensel's Lemma used in \cite{knapp2016solubility} and \cite{duncan2020solubility}.

%Hensel's Lemma gives conditions under which a nontrivial zero modulo a power of $\pi$ implies the existence of a nontrivial zero over $K$.  This power $\gamma$ is given by the formula $\gamma = 
%These conditions depend upon the power $\tau$ of $2$ dividing the degree $d$ of the form, and on the degree of ramification $e$ of $K / \mathbb{Q}_2$.  Because we share these in common with the conditions investigated in \cite{knapp2016solubility}, we give the same statement.  (For a more general discussion of Hensel’s Lemma, see \cite{greenberg1969lectures}.)

\begin{lemma}[Hensel's Lemma]
Let $d=2m$, $m$ odd, and $x_i$ be a variable of (\ref{eq}) at level $h$.  Suppose that $x_i$ can be used in a contraction of variables (or one in a series of contractions) which produces a new variable at level at least $h+5$.  Then (\ref{eq}) has a nontrivial zero.
\end{lemma}

By Hensel's Lemma, for the rest of this paper we may assume that an \textit{s3}-contraction produces a variable either 3 or 4 levels higher, and a \textit{t}-contraction produces a variable exactly 4 levels higher.  To draw the reader's attention to a level from which a variable satisfying this statement of Hensel's Lemma will originate, we append an asterisk.  For example:

$$({}^{2}_{0}, 0, {}^{1}_{1}) \xrightarrow{s2} (0, 0, {}^{2*}_{1})$$

\section{General Lemmas}

\begin{lemma} \label{slide}
Suppose that (\ref{eq}) has two variables in level $k$, and at least one variable in levels $k+1$, $k+2$, ... $k+t-1$, then contractions can be performed to produce a variable at level at least $k+t$.
\end{lemma}
\begin{proof}
Any two variables in the same level can be contracted to a variable at least one level higher.  By repeated contractions, we obtain the desired variable.
\end{proof}

\begin{lemma}\label{fromtwo}
Suppose that after a series of contractions, (\ref{eq}) has two variables in the same level with the same $\pi$-coefficient, one of which has been raised at least two levels, or it has two variables in the same level, one of which has been raised at least four levels.  Then (\ref{eq}) has a nontrivial zero.
\end{lemma}
\begin{proof}
If one of the variables has been raised at least two levels, performing an \textit{s3}-contraction with the two variables results in a variable that has been raised at least an additional three levels, for a total of at least five levels.  If one of the variables has been raised at least four levels, performing any contraction results in a variable that has been raised at least one additional level, for a total of at least five levels.  A zero follows from Hensel's Lemma.
\end{proof}

\begin{lemma}\label{empty1}
Suppose that (\ref{eq}) has at least four variables in level $k$ which can be used to form a pair with the same $\pi$-coefficient and a pair with differing $\pi$-coefficients, and at least one variable in level $k+1$.  Then (\ref{eq}) has a nontrivial zero.
\end{lemma}
\begin{proof}
$({}^{3}_{1}, 1) \xrightarrow{s2} ({}^{1}_{1}, 1, 1) \xrightarrow{d} (0, {}^{1*}_{1}, 1) \xrightarrow{d} (0, 0, {}^{2*}_{0})$

A zero follows from Lemma \ref{fromtwo}.
\end{proof}

\begin{lemma}\label{empty4}
Suppose $K \in \{\mathbb{Q}_2(\sqrt{-1}), \mathbb{Q}_2(\sqrt{-5})\}$, that (\ref{eq}) has at least three variables in level $k$  with the same $\pi$-coefficient, and at least one variable in level $k+4$.  Then (\ref{eq}) has a nontrivial zero.
\end{lemma}
\begin{proof}
$({}^{3}_{0}, 0, 0, 0, 1) \xrightarrow{t} (1, 0, 0, 0, 2*)$

A zero follows from Lemma \ref{fromtwo}.
\end{proof}

\begin{lemma} \label{empty234}
Suppose that (\ref{eq}) has at least four variables in level $k$ which can be used to form two pairs with the same $\pi$-coefficient, and a variable in at least one of levels $k+2$, $k+3$, or $k+4$.  Then (\ref{eq}) has a nontrivial zero.
\end{lemma}
\begin{proof}
\begin{itemize}
    \item $({}^{2}_{2}, 0, 1) \xrightarrow{s2, s2} (0, 0, 3*)$
    
    A zero follows from Lemma \ref{fromtwo}.
    
    \item $({}^{2}_{2}, 0, 0 ,1) \xrightarrow{s2,s2}$ 
    \begin{itemize}
        \item $(0, 0, {}^{2*}_{0}, 1)$
        \item $(0, 0, {}^{1*}_{1*}, 1) \xrightarrow{d} (0, 0, 0, {}^{2*}_{0})$
    \end{itemize}
    
    A zero follows in both cases from Lemma \ref{fromtwo}.
    
    In this last case, assume by Lemma \ref{fromtwo} that the variables resulting from two \textit{s3}-contractions are not produced in the same level.
    
    \item $({}^{2}_{2}, 0, 0 ,0, 1) \xrightarrow{s3, s3} (0, 0, 0, 1*, 2*)$
    
    A zero follows from Lemma \ref{fromtwo}.
\end{itemize}

\end{proof}

\section{Lemmas Regarding Large Numbers of Variables}

\begin{lemma}\label{max7}
Suppose (\ref{eq}) has at least six variables in the same level  which can be used to form three pairs with the same $\pi$-coefficient. Then (\ref{eq}) has a nontrivial zero.

In particular, if (\ref{eq}) has seven variables in the same level, then (\ref{eq}) has a nontrivial zero.
\end{lemma}
\begin{proof}
If there are seven variables in the same level, by the pigeonhole principle, three pairs of variables may be formed having the same $\pi$-coefficient.  After performing an \textit{s2}-contraction with one of these pairs, a zero follows from Lemma \ref{empty234}.
\end{proof}

\begin{corollary}
Artin's Additive Form Conjecture holds for all quadratic ramified extensions of $\mathbb{Q}_2$ and degrees $d=2m$, $m$ odd.
\end{corollary}
\begin{proof}
The case $m=1$ is treated separately; see \cite{MR2104929}, Chapter VI.  In the cases considered, $d \ge 6$, and so $d^2 + 1 \ge 6d + 1$, and by the pigeonhole principle, the form has a level with at least seven variables.  A zero follows from Lemma \ref{max7}.
\end{proof}

\begin{lemma} \label{fivesame}
Suppose $K \in \{\mathbb{Q}_2(\sqrt{-1}), \mathbb{Q}_2(\sqrt{-5})\}$ and (\ref{eq}) has at least five variables in the same level with the same $\pi$-coefficient, or at least six variables with at least three variables in each $\pi$-coefficient class.  Then (\ref{eq}) has a nontrivial zero.
\end{lemma}
\begin{proof}~

$({}^{5}_{0}) \xrightarrow{t, t} (1,0,0,0,2*)$

$({}^{3}_{3}) \xrightarrow{t, t} ({}^{1}_{1},0,0,0,2*)$

A zero follows from Lemma \ref{fromtwo}.
\end{proof}

\begin{lemma} \label{max6}
Suppose that $s \ge \frac{7}{5}d$ and after normalization (\ref{eq}) has at least six variables in level 0.  Then (\ref{eq}) has a nontrivial zero.  

If $K \in \{\mathbb{Q}_2(\sqrt{-1}), \mathbb{Q}_2(\sqrt{-5})\}$, then the condition on $s$ is unnecessary.
\end{lemma}
\begin{proof}
By Lemma \ref{max7}, assume that level 0 contains exactly six variables.  By Lemma \ref{empty234}, assume levels 2, 3, and 4 are unoccupied.  By normalization, $s \ge \frac{7}{5}d$ implies that levels 0 through 4 together contain at least seven variables, and so level 1 contains at least one variable.  By Lemma \ref{empty1}, assume all of the variables in level 0 have the same $\pi$-coefficient.  A zero follows from Lemma \ref{max7}.

Now, suppose that $K \in \{\mathbb{Q}_2(\sqrt{-1}), \mathbb{Q}_2(\sqrt{-5})\}$.  By Lemma \ref{fivesame}, assume that four of the variables in level 0 are in one $\pi$-coefficient class, and two are in the other.  A zero follows from Lemma \ref{max7}.
\end{proof}

\begin{lemma} \label{max5}
Suppose that $s \ge \frac{7}{5}d$ and after normalization (\ref{eq}) has at least five variables in level 0.  Then (\ref{eq}) has a nontrivial zero.  

If $K \in \{\mathbb{Q}_2(\sqrt{-1}), \mathbb{Q}_2(\sqrt{-5})\}$, then $s \ge d + 1$ is sufficient.
\end{lemma}
\begin{proof}
By Lemma \ref{max6}, assume level 0 contains exactly five variables.  By Lemma \ref{empty234}, assume levels 2, 3, and 4 are unoccupied.

By normalization, $s \ge \frac{7}{5}d$ implies that levels 0 through 4 together contain at least seven variables, and so level 1 contains at least two variables.  These variables can be contracted to a variable in level 2 or 3 using a \textit{d}- or \textit{s2}-contraction, and a zero follows from Lemma \ref{empty234}.

Now, suppose that $K \in \{\mathbb{Q}_2(\sqrt{-1}), \mathbb{Q}_2(\sqrt{-5})\}$.  By normalization, $s \ge d + 1$ implies that levels 0 through 4 together contain at least six variables, and so level 1 contains at least one variable. By Lemma \ref{empty1}, assume all the variables in level 0 have the same $\pi$-coefficient.  A zero follows from Lemma \ref{fivesame}.
\end{proof}

\begin{lemma} \label{max4}
Suppose that $s \ge \frac{7}{5}d$ and after normalization (\ref{eq}) has at least four variables in level 0 which can be used to form two pairs with the same $\pi$-coefficient.  Then (\ref{eq}) has a nontrivial zero.  

If $K \in \{\mathbb{Q}_2(\sqrt{-1}), \mathbb{Q}_2(\sqrt{-5})\}$, then $s \ge d + 1$ is sufficient.
\end{lemma}
\begin{proof}
By Lemma \ref{max5}, assume level 0 contains exactly four variables.  By Lemma \ref{empty234}, assume levels 2, 3, and 4 are unoccupied. By normalization, $s \ge d + 1$ implies that levels 0 through 4 together contain at least six variables, and so level 1 contains at least two variables.  Any two of the variables in level 1 can be contracted to a variable in level 2 or 3, and a zero follows from Lemma \ref{empty234}.
\end{proof}

\section{Lemmas Regarding Small Numbers of Variables}

\begin{lemma}\label{2same2diff}
Suppose that (\ref{eq}) has two variables in level $k$ with the same $\pi$-coefficient, and two variables in at least one of levels $k+1$, $k+2$, or $k+3$ with differing $\pi$-coefficients.  Then (\ref{eq}) has a nontrivial zero.
\end{lemma}
\begin{proof}~

\begin{itemize}
    \item $({}^{2}_{0}, {}^{1}_{1}) \xrightarrow{s2} (0, {}^{1}_{1}, 1*) \xrightarrow{d} (0, 0, {}^{2*}_{0})$
    \item $({}^{2}_{0}, 0, {}^{1}_{1}) \xrightarrow{s2} (0, 0, {}^{2*}_{1})$
    \item $({}^{2}_{0}, 0, 0, {}^{1}_{1}) \xrightarrow{s3}$
    \begin{itemize}
        \item $(0, 0, 0, {}^{2*}_{1})$
        \item $(0, 0, 0, {}^{1}_{1}, 1*) \xrightarrow{d} (0, 0, 0, 0, {}^{2*}_{0})$
    \end{itemize}
\end{itemize}

In each case, a zero follows from Lemma \ref{fromtwo}.
\end{proof}

\begin{lemma}\label{2same2more}
Suppose that (\ref{eq}) has two variables in level $k$ with the same $\pi$-coefficient, and a variable in both of levels $k+2$ and $k+3$, or both of levels $k+3$ and $k+4$.  Then (\ref{eq}) has a nontrivial zero.
\end{lemma}
\begin{proof}~
\begin{itemize}
    \item $({}^{2}_{0}, 0, 1, 1) \xrightarrow{s2}$
    \begin{itemize}
        \item $(0, 0, {}^{2*}_{0}, 1)$
        \item $(0, 0, {}^{1*}_{1}, 1) \xrightarrow{d} (0, 0, 0, {}^{2*}_{0})$
    \end{itemize}

    A zero follows from Lemma \ref{fromtwo}.
    \item $({}^{2}_{0}, 0, 0, 1, 1) \xrightarrow{s3}$
    \begin{itemize}
        \item $(0, 0, 0,2*, 1)$
        \item $(0, 0, 0,1, 2*)$
    \end{itemize}
    
    A zero follows from Lemma \ref{slide} and Hensel's Lemma.
\end{itemize}
\end{proof}

\begin{lemma} \label{2and2and1}
Suppose that (\ref{eq}) has two variables in level $k$ with the same $\pi$-coefficient, at least two variables in level $k+1$, and a variable in at least one of level $k+2$, $k+3$, or $k+4$.  Then (\ref{eq}) has a nontrivial zero.
\end{lemma}
\begin{proof}
By Lemma \ref{2same2diff}, assume the variables in level $k+1$ have the same $\pi$-coefficient.  They may be contracted to a variable in level $k+3$, and so if there is a variable in level $k+2$ or $k+4$, a zero follows from Lemma \ref{2same2more}.  If there is a variable in level $k+3$, by Lemma \ref{2same2diff}, assume it has the same $\pi$-coefficient as the resulting variable.  A zero follows from Lemma \ref{fromtwo}.
\end{proof}

\begin{lemma}\label{2diff2more}
Suppose that (\ref{eq}) has two variables in level $k$ with differing $\pi$-coefficients, a variable in level $k+1$, and a variable in at least one of levels $k+2$ or $k+4$.  Then (\ref{eq}) has a nontrivial zero.

\end{lemma}
\begin{proof}~
\begin{itemize}
    \item $({}^{1}_{1}, 1, 1) \xrightarrow{d} (0, {}^{1*}_{1}, 1) \xrightarrow{d} (0, 0, {}^{2*}_{0})$
    \item $({}^{1}_{1}, 1, 0, 0, 1) \xrightarrow{d} (0, {}^{2}_{0}, 0, 0, 1) \xrightarrow{s3} (0, 0, 0, 0, 2*)$
    
In both cases, a zero follows from Lemma \ref{fromtwo}.
\end{itemize}
\end{proof}

\begin{lemma} \label{threeokay}
Suppose that $K \in \{\mathbb{Q}_2(\sqrt{2}), \mathbb{Q}_2(\sqrt{10}), \mathbb{Q}_2(\sqrt{-2}), \mathbb{Q}_2(\sqrt{-10})\}$, $s \ge \frac{10}{7}d$, and after normalization, for some $k$ with $0 \le k \le d-4$, (\ref{eq}) has two variables in level $k$ with the same $\pi$-coefficient, at least four variables in level $k+2$, and at most $k+4$ variables in levels 0 through $k+1$ together.  Then (\ref{eq}) has a nontrivial zero.
\end{lemma}
\begin{proof}
By Lemma \ref{2same2diff}, assume all of the variables in level $k+2$ have the same $\pi$-coefficient.  By Lemma \ref{2same2more}, assume level $k+3$ is unoccupied.  By normalization, level 0 is occupied, and so if any of levels $k+4$, $k+5$, and $k+6$ is congruent to $0$ modulo $d$, then a zero follows from Lemma \ref{empty234}. Thus, assume $k+6 < d$.  By Lemma \ref{empty234}, assume levels $k+4$, $k+5$, and $k+6$ are unoccupied.  By normalization, $s \ge \frac{10}{7}d$ implies that levels 0 through $k+6$ contain at least $\frac{10}{7}(k+7) \ge k+10$ variables, and so level $k+2$ contains at least six variables.  A zero follows from Lemma \ref{max7}.
\end{proof}

\begin{lemma} \label{max3}
Suppose that $K \in \{\mathbb{Q}_2(\sqrt{2}), \mathbb{Q}_2(\sqrt{10}), \mathbb{Q}_2(\sqrt{-2}), \mathbb{Q}_2(\sqrt{-10})\}$, $s \ge \frac{3}{2}d$.  Then (\ref{eq}) has a nontrivial zero.
\end{lemma}
\begin{proof}
Assume that the form is normalized.  By Lemma \ref{max5}, assume that level 0 has at most four variables.  

%We will examine each possibility for how the variables may be distributed between levels 0 and 1 and their $\pi$-coefficient classes given these assumptions, and conclude how the variables in levels 2 and 3 (and in some cases, additional higher levels) can be distributed.  Then, given how the variables are assumed to be distributed in levels 0 through $k$, we will determine how they may be distributed between levels $k+1$ and $k+2$.  From this we will conclude that in each case that the form has a nontrivial zero.

First, suppose that level 0 contains four variables.  By Lemma \ref{max4}, assume that three of the variables fall into one $\pi$-coefficient class, and one into the other.  By Lemma \ref{empty1}, assume level 1 is unoccupied.  Therefore in this case, the form begins $[{}^{3}_{1},0]$.

%By normalization, $s \ge \frac{3}{2}d$ implies that levels 0 through 2 together contain at least five variables, and so level 2 contains at least one variable.  By Lemma \ref{2same2more}, assume level 3 is unoccupied, and so by normalization, level 2 contains at least two variables, and by Lemma \ref{threeokay}, at most three.  By Lemma \ref{2same2diff}, assume they have the same $\pi$-coefficient.  

%followed by either $[{}^2_0,0]$ or $[{}^3_0,0]$.

Now, suppose level 0 contains exactly three variables, not all having the same $\pi$-coefficient.  Suppose that level 1 is occupied.  By Lemma \ref{2diff2more}, assume level 2 is unoccupied.  By normalization, level 1 contains at least two variables.  By Lemma \ref{2and2and1}, assume level 3 is unoccupied.  By normalization, level 1 contains at least three variables, and by Lemma \ref{2same2diff} assume they have the same $\pi$-coefficient.

$$({}^{2}_{1}, {}^{3}_{0}) \xrightarrow{d} (1, {}^{4*}_{0}) \xrightarrow{s2, s2} (1, 0, 0, 2*)$$
A zero follows from Lemma \ref{fromtwo}.  Thus assume level 1 is unoccupied.  Therefore in this case, the form begins $[{}^2_1,0]$.

%By normalization, level 2 is occupied, and by Lemma \ref{2same2more}, assume level 3 is unoccupied.  By normalization, then, level 2 contains at least 3 variables. By Lemma \ref{threeokay}, assume there are exactly three, and by Lemma \ref{2same2diff}, assume they all have the same $\pi$-coefficient.  

%By Lemma \ref{max7}, assume there are at most five variables in level 2.  By normalization, level 4 is occupied, and so by Lemma \ref{empty234}, assume level 2 contains exactly 3 variables.  One may follow the same reasoning to see that level 4 contains exactly three variables having the same $\pi$-coefficient, level 5 is unoccupied, etc., so that the distribution of all of the variables among levels and $\pi$-coefficient classes is determined.  The form is assumed then to be of type $[{}^{2}_{1}, 0, {}^{3}_{0}, 0, {}^{3}_{0}, 0, \ldots, {}^{3}_{0}, 0]$, and a zero follows from Lemma \ref{2same2diff}.

Next, suppose that level 0 contains exactly three variables,  all having the same $\pi$-coefficient.  First, suppose that level 1 contains at least two variables.  By Lemma \ref{2and2and1}, assume levels 2, 3, and 4 are unoccupied.  By normalization, level 1 contains at least five variables, a pair of which can be contracted up exactly two level, and a zero follow from Lemma \ref{2and2and1}.  Thus, assume level 1 contains at most one variable.  Therefore in this case, the form begins $[{}^3_0,1]$ or $[{}^3_0,0]$.

%Suppose level 1 contains exactly one variable.  By normalization, level 2 contains a variable, and so by Lemma \ref{2same2more}, assume level 3 is unoccupied.  Thus level 2 contains at least two variables, and by Lemma \ref{threeokay}, assume there are at most three.  By Lemma \ref{2same2diff}, assume they have the same $\pi$-coefficient.

%Now, suppose level 1 is unoccupied.  By normalization, level 2 is occupied, and by Lemma \ref{2same2more}, assume level 3 is unoccupied.  By normalization, level 2 contains at least three variables, and by Lemmas \ref{threeokay} and \ref{2same2diff}, assume it contains exactly three, all having the same $\pi$-coefficient.

%By very similar reasoning, the form is of type $[{}^{3}_{0}, 0, {}^{3}_{0}, 0, \ldots]$.  If there are two consecutive even numbered levels where (all six of) the variables have the same $\pi$-coefficient, then perform an \textit{st}-contraction from the former to the latter, and a zero follows.  By the pigeonhole principle, this must be the case since there are $m$ even numbered levels, and $m$ is odd.
%\end{proof}

%\begin{lemma} \label{max2}
%Suppose that $K \in \{\mathbb{Q}_2(\sqrt{2}), \mathbb{Q}_2(\sqrt{10}), \mathbb{Q}_2(\sqrt{-2}), \mathbb{Q}_2(\sqrt{-10})\}$, $s \ge \frac{3}{2}d$, and after normalization (\ref{eq}) has at least two variables in level 0.  Then (\ref{eq}) has a nontrivial zero.
%\end{lemma}
%\begin{proof}

Finally, suppose that level 0 contains exactly two variables. By normalization level 1 is occupied.  Suppose the variables in level 0 have differing $\pi$-coefficients.  By Lemma \ref{2diff2more}, assume level 2 is unoccupied, and so by normalization, level 1 contains at least three variables.  The two variables in level 0 can be contracted to a variable in level 1 so that its resulting four variables can be used to form two pairs with the same $\pi$-coefficient.  By Lemma \ref{empty234} then, assume levels 3, 4, and 5 are unoccupied, and so by normalization level 1 contains at least seven variables.  A zero follows from Lemma \ref{max7}.  Thus, assume the variables in level 0 have the same $\pi$-coefficient.  Suppose level 1 contains at least two variables.  By Lemma \ref{2same2diff}, assume they have the same $\pi$-coefficient.  By Lemma \ref{2and2and1}, assume levels 2, 3, and 4 are unoccupied. and so by normalization level 1 contains at least six variables with the same $\pi$-coefficient, and a zero follows from Lemma \ref{max7}.  Thus assume level 1 contains exactly one variable.  Therefore in this case, the form begins $[{}^2_0,1]$.

%By normalization, level 2 is occupied and so by Lemma \ref{2same2more}, assume level 3 is unoccupied; by normalization level 2 contains at least three variables.  By Lemma \ref{threeokay}, assume it contains exactly three variables, having the same $\pi$-coefficient by Lemma \ref{2same2diff}.

It follows from the above that the form begins with one of the blocks $[{}^{2}_{1}, 0]$, $[{}^{3}_{0}, 0]$, $[{}^{2}_{0}, 1]$ , $[{}^{3}_{1},0]$, or $[{}^{3}_{0}, 1]$.  In each case there are either $\frac{3}{2}\ell$ or  $\frac{3}{2} \ell + 1$ variables in the $\ell = 2$ levels specified.  We will show by induction that the subsequent pairs of levels are composed of the blocks $[{}^2_0,0]$ and $[{}^3_0,0]$.  By normalization, level $\ell$ is occupied, by Lemma \ref{2same2diff} assume level $\ell + 1$ is unoccupied, and so by normalization and Lemmas \ref{threeokay} and \ref{2same2diff}, assume level $\ell$ contains two or three variables having the same $\pi$-coefficient.  If $\frac{3}{2}\ell$ variables are contained in levels 0 through $\ell - 1$, then by normalization level $\ell$ contains exactly three.  By induction, it follows that level $d-2$ contains at least two variables with the same $\pi$-coefficient, and so if levels 0 and 1 are both occupied, a zero follows from Lemma \ref{2same2more} with $k=d-2$, and if level 0 contains two variables with differing $\pi$-coefficients, then a zero follows from Lemma \ref{2same2diff} with $k=d-2$.

The only case remaining is forms that begin $[{}^{3}_{0}, 0, {}^{3}_{0}, 0, {}^{3}_{0}, 0, \ldots]$.  If for some even $k$ the variables in \textit{both} of levels $k$ and $k+2$ are all in the same $\pi$-coefficient class, then perform an \textit{st}-contraction from level $k$ to level $k+2$ and a zero follows from Lemma \ref{fromtwo}.  Thus assume the $\pi$-coefficient classes of the variables in the even numbered levels alternate.  Using a slight modification of our notation, this might be written $[{}^{3}_{0}, 0, {}^{0}_{3}, 0, {}^{3}_{0}, 0, {}^{0}_{3}, \ldots]$.  However, because $d=2m$, where $m$ is odd, the variables in levels $d-2$ and $0$ all have the same $\pi$-coefficient, and a zero follows.

%Suppose level 2 contains at least four variables.  Then by Lemma \ref{empty234}, assume levels 5, 6, and 7 are unoccupied and by normalization level 2 contains at least seven variables and a zero follows from Lemma \ref{max7}.  Thus assume level 2 contains exactly three variables, and by normalization level 4 contains at least three variables.  By the same argument, it contains exactly three variables, and this pattern $[{}^{2}_{0},1,{}^{3}_{0},0,{}^{3}_{0},0,\ldots]$ repeats for the rest of the form.  A zero follows from Lemma \ref{2same2more}.
\end{proof}

\begin{lemma} \label{max3again}
Suppose that $K \in \{\mathbb{Q}_2(\sqrt{-1}), \mathbb{Q}_2(\sqrt{-5})\}$, $s \ge d+1$ and after normalization (\ref{eq}) has at least three variables in level 0 and at least two variables in level 1.  Then (\ref{eq}) has a nontrivial zero.
\end{lemma}
\begin{proof}
 By Lemmas \ref{empty1} and \ref{max4}, assume level 0 contains exactly three variables.  By Lemma \ref{2same2diff}, assume the variables in level 1 have the same $\pi$-coefficient.  By Lemma \ref{2and2and1}, assume levels 2, 3, and 4 are unoccupied.  By normalization, levels 0 through 4 together contain at least six variables, and so level 1 contains at least three variables having the same $\pi$-coefficient. Thus by Lemma \ref{empty4} with $k=1$ assume level 5 is unoccupied.  By normalization, levels 0 through 5 together contain at least seven variables, and so level 1 contains at least four variables with the same $\pi$-coefficient.  Contract both pairs to level 3, and by Lemma \ref{2same2diff}, assume the resulting variables have the same $\pi$-coefficients.  A zero follows from Lemma \ref{fromtwo}.
\end{proof}

\begin{lemma} \label{threeoh}
Suppose that $K \in \{\mathbb{Q}_2(\sqrt{-1}), \mathbb{Q}_2(\sqrt{-5})\}$, $s \ge d+1$ and after normalization, for some $i \in \{2,3\}$ and $k$ with $0 \le k+i \le d-1$, (\ref{eq}) has two variables in level $k$ with the same $\pi$-coefficient, at least four variables in level $k+i$, no variables in level $k+i + 1$, and at most $k+i+1$ variables in levels 0 through $k+i-1$ together.  Then (\ref{eq}) has a nontrivial zero.
\end{lemma}
\begin{proof}
By Lemma \ref{2same2diff}, assume all of the variables in level $k+i$ have the same $\pi$-coefficient.  If any of $k+i+2$, $k+i+3$, or $k+i+4$ is congruent to 0 modulo $d$, then a zero follows from Lemma \ref{empty234}.  Otherwise, assume these levels are unoccupied, and so by normalization, $s \ge d+1$ implies that levels 0 through $k+i+4$ contain at least $k+i+6$ variables, and so level $k+i$ contains at least five variables.  A zero follows from Lemma \ref{fivesame}.
\end{proof}

\begin{lemma} \label{final}
Suppose that $K \in \{\mathbb{Q}_2(\sqrt{-1}), \mathbb{Q}_2(\sqrt{-5})\}$ and $s \ge d+1$.  Then (\ref{eq}) has a nontrivial zero.
\end{lemma}
\begin{proof}
Assume that the form is normalized.  By Lemma \ref{max5}, assume that level 0 has at most four variables.

First, suppose level 0 has four variables.  By Lemma \ref{max4}, three of them fall into one $\pi$-coefficient class, and one into the other.  By Lemmas \ref{empty1} and \ref{empty4}, assume levels 1 and 4 are unoccupied.  By normalization, levels 2 and 3 together contain at least two variables, and by Lemma \ref{2same2more}, assume at most one is occupied.  By Lemma \ref{2same2diff}, assume they have the same $\pi$-coefficient.  If level 2 contains at least two variables, they can be contracted to a variable in level 4 and a zero follows from Lemma \ref{empty4}.  Thus assume level 2 is unoccupied, level 3 contains at least two variables, and by Lemma \ref{threeoh} at most three variables.  Therefore in this case, the form begins with $[{}^3_1,0,0]$ followed by $[{}^2_0,0]$ or $[{}^3_0,0]$.

%Therefore in this case, the form begins $[{}^3_1, 0]$ followed by $[{}^2_0,0,0]$ or $[{}^{3}_{0},0,0]$.

Next, suppose there are exactly three variables in level 0, not all having the same $\pi$-coefficient.  By Lemma \ref{max3again}, assume level 1 has at most one variable.  Suppose level 1 has exactly one variable.  Then by Lemma \ref{2diff2more}, assume levels 2 and 4 are unoccupied.  By normalization, level 3 contains at least two variables, and by Lemma \ref{threeoh} at most three.  By Lemma \ref{2same2diff} assume they have the same $\pi$-coefficient.  Thus assume level 3 contains two or three variables having the same $\pi$-coefficient, and level 4 contains zero.  Now, suppose level 1 is unoccupied.  By normalization level 2 is occupied, and by Lemma \ref{2same2more}, assume level 3 is unoccupied.  By normalization, level 2 contains at least two variables, and by Lemma \ref{threeoh} at most three.  By Lemma \ref{2same2diff} assume they have the same $\pi$-coefficient.  Thus assume level 2 contains two or three variables and level 3 contains zero.  Therefore in this case, the form begins with $[{}^2_1,1,0]$ or $[{}^2_1, 0]$ followed by $[{}^2_0,0]$ or $[{}^3_0,0]$.

Now suppose level 0 contains exactly three variables having the same $\pi$-coefficient.  By Lemma \ref{max3again}, assume level 1 has at most one variable.  Suppose level 1 is unoccupied.  By normalization level 2 contains at least one variable, and so by Lemma \ref{2same2more}, assume level 3 is unoccupied.  By normalization, level 2 contains at least two variables.  By Lemma \ref{2same2diff}, assume they have the same $\pi$-coefficient and thus can be contracted to a variable in level 4.  A zero follows from Lemma \ref{empty4}.  Thus assume level 1 contains exactly one variable.  By Lemma \ref{empty4}, assume level 4 is unoccupied.  By normalization, levels 2 and 3 together contain at least two variables, and by Lemma \ref{2same2more}, assume at most one is occupied, and by Lemma \ref{2same2diff} that they have the same $\pi$-coefficient.  If the variables are in level 2, they can be contracted to a variable in level 4, and a zero follows from Lemma \ref{empty4}.  Thus, by Lemma \ref{threeoh} assume level 3 contains two or three variables with the same $\pi$-coefficient and level 4 contains zero.  Therefore in this case, the form begins $[{}^3_0, 1, 0]$ followed by $[{}^2_0,0]$ or $[{}^3_0,0]$.

Next, suppose level 0 contains exactly two variables with differing $\pi$-coefficients.  By normalization, level 1 contains at least one variable, and by Lemma \ref{2diff2more}, assume levels 2 and 4 are unoccupied, and so by normalization level 1 contains at least two variables.  First, suppose level 1 contains at least three variables.  A \textit{d}-contraction can be performed from level 0 so that the resulting variables in level 1 form two pairs with the same $\pi$-coefficients.  By Lemma \ref{empty234} with $k=1$, assume levels 3, 4, and 5 are also unoccupied.  By normalization level 1 contains at least five variables.  A \textit{d}-contraction can be performed from level 0 such that the resulting six variables in level 1 can be used to form three pairs each having the same $\pi$-coefficient; a zero follows from Lemma \ref{max7}.  Thus, assume level 1 contains exactly two variables, and so by normalization applied to levels 0 through 4, level 3 contains at least two variables.  After a \textit{d}-contraction from level 0, assume by Lemma \ref{2same2diff} with $k=1$ that the variables in level 3 have the same $\pi$-coefficient.  Suppose level 3 contains at least four variables.  By Lemma \ref{empty234}, if $d=6$ a zero follows; otherwise assume levels 5, 6, and 7 are unoccupied.  By normalization level 3 contains at least five variables, and a zero follows from Lemma \ref{fivesame}.  Thus assume level 3 contains at most three variables.  Therefore in this case, the form begins with $[{}^1_1, 2, 0]$ and is followed by $[{}^2_0, 0]$ or $[{}^3_0, 0]$.

%Suppose they are not all in the same $\pi$-coefficient class.
%$({}^{1}_{1}, 2, 0, {}^{1}_{1}) \xrightarrow{d} (0, {}^{2*}_{1}, 0, {}^{1}_{1}) \xrightarrow{s2} (0, 1, 0, {}^{2*}_{1})$  A zero follow from Lemma \ref{fromtwo}.  By Lemma \ref{threeoh}, level 3 contains at most three variables.  Thus, assume level 3 contains two or three variables, all having the same $\pi$-coefficient, and level 4 is unoccupied.

Finally, suppose level 0 contains exactly two variables with the same $\pi$-coefficient.  By normalization, level 1 is occupied. Suppose level 1 contains at least two variables.  By Lemma \ref{2and2and1}, assume levels 2, 3, and 4 are unoccupied.
%  If level 2 contains a variable, by Lemma \ref{2same2diff} assume that an \textit{s2}-contraction from level 0 produces a variable with the same $\pi$-coefficient; a zero follows from Lemma \ref{fromtwo}.  If level 3 contains a variable, by Lemma \ref{2same2diff} assume that an \textit{s2}-contraction from level 1 produces a variable with the same $\pi$-coefficient; a zero follows from Lemma \ref{fromtwo}.  If level 4 contains a variable, by Lemma \ref{fromtwo} assume an \textit{s3}-contraction from level 0 produces a variable in level 3; a zero follows from Lemma \ref{2same2more} with $k=1$.
By normalization, level 1 contains at least four variables.  By Lemma \ref{empty4} with $k=1$, assume level 5 is unoccupied, and so by normalization level 1 contains at least five variables.  By Lemma \ref{2same2diff}, assume they have the same $\pi$-coefficient; a zero follows from Lemma \ref{fivesame}. Thus assume level 1 contains exactly one variable, and by normalization level 2 is occupied.  By Lemma \ref{2same2more}, assume level 3 is unoccupied, and so by normalization level 2 contains at least two variables, and by Lemma \ref{threeoh}, at most three.  By Lemma \ref{2same2diff}, assume they have the same $\pi$-coefficient.  Therefore in this case, the form begins with $[{}^2_0, 1]$ followed by $[{}^2_0,0]$ or $[{}^3_0,0]$.

%Then, by Lemma \ref{2and2and1}, assume levels 2, 3, and 4 are unoccupied, and so by normalization level 1 contains at least three variables.  , and , and so by Lemma \ref{threeoh}, level 3 contains two or three variables having the same $\pi$-coefficient, and level 4 is unoccupied.

It follows from the above that the form begins with $[{}^3_1,0,0]$, $[{}^2_1,1,0]$, $[{}^2_1,0]$, $[{}^3_0,1,0]$, $[{}^1_1,2,0]$, or $[{}^2_0, 1]$ followed by $[{}^2_0, 0]$ or $[{}^3_0, 0]$.  In each case, the form begins with a block of $\ell = 2$ or $\ell = 3$ levels, level $\ell$ is assumed to have either two or three variables having the same $\pi$-coefficient, level $\ell + 1$ is assumed to be unoccupied, and levels 0 through $\ell - 1$ contain exactly $\ell + 1$ variables.  If $\ell = d-2$ or $\ell = d-3$, a zero follows from the argument below.  Otherwise, we proceed by induction as follows.

First suppose that level $\ell$ contains exactly two variables.  By normalization, level $\ell + 2$ is occupied, and by Lemma \ref{2same2more} with $k = \ell$, assume level $\ell + 3$ is unoccupied.  By normalization, level $\ell + 2$ contains at least two variables, by Lemma \ref{threeoh} with $k = \ell$ and $i=2$ at most three, and by Lemma \ref{2same2diff} with $k = \ell$, they have the same $\pi$-coefficient.  

Now, suppose that level $\ell$ contains three variables.  By Lemma \ref{empty4}, assume level $\ell+4$ is unoccupied, and by Lemma \ref{2same2diff}, assume the variables in each of levels $\ell+2$ and $\ell+3$ have the same $\pi$-coefficient.  By normalization, levels $\ell+2$ and $\ell+3$ together contain at least two variables, and by Lemma \ref{2same2more} only one is occupied.  If the variables are in level $\ell+2$, they can be contracted to a variable in level $\ell+4$ and a zero follows from Lemma \ref{empty4}.  Thus assume level $\ell+2$ is unoccupied and by Lemma \ref{threeoh}, assume level $\ell+3$ contains two or three variables.

In both cases, the initial block of $\ell$ levels is followed by $[{}^2_0, 0]$ or $[{}^3_0,0,0]$.  Further, the next two levels are subject to the same constraints as were levels $\ell$ and $\ell + 1$, and so by induction it follows that at least one of levels $d-3$ or $d-2$ contains two variables with the same $\pi$-coefficient, or level $d-4$ contains three variables with the same $\pi$-coefficient.  In the former case, a zero follows from Lemmas \ref{2same2more} and \ref{2same2diff} with $k = d-3$ or $k=d-2$.  In the latter case, a zero follows from Lemma \ref{empty4} with $k=d-4$.

%there is a level $k$ assumed to contain two or three variables all having the same $\pi$-coefficient, $k+1$ containing none, and a total of exactly $k+1$ variables in levels 0 through $k-1$.  Suppose there are two variables in level $k$.  By normalization level $k+2$ is occupied, and by Lemma \ref{2same2more}, assume level $k+3$ is unoccupied, and so by normalization and Lemma \ref{threeoh}, level $k+2$ contains two or three variables.  Suppose there are three variables in level $k$.  

%In both cases, the hypotheses at the beginning of the previous paragraph are again satisfied.  Thus it may be assumed that the type of the form begins with $[{}^{2}_{1}, 1, 0]$, $[{}^{2}_{1}, 0]$,  $[{}^{3}_{0}, 1, 0]$, $[{}^{1}_{1}, 2, 0]$, or $[{}^{2}_{0}, 1]$ and is followed by blocks of the form $[{}^{2}_{0},0]$ or $[{}^{3}_{0}, 0, 0]$.  A zero follows from Lemmas \ref{2same2diff} and \ref{2same2more}.
\end{proof}

If $K \in \{\mathbb{Q}_2(\sqrt{2}), \mathbb{Q}_2(\sqrt{10}), \mathbb{Q}_2(\sqrt{-2}), \mathbb{Q}_2(\sqrt{-10})\}$, from Lemma \ref{max3} we have $\Gamma(d, K) \le \frac{3}{2}d$.  To demonstrate we have the equality from the first part of Theorem \ref{theo}, we present a form in $s = \frac{3}{2}d - 1$ variables that has no nontrivial zero.

\begin{align*}
f = \sum_{i=0}^{\frac{d-2}{4} - 1} \pi^{4i}[(x_{6i}^d + x_{6i+1}^d + x_{6i+2}^d) + (\pi^2+\pi^3)(x_{6i+3}^d + x_{6i+4}^d + x_{6i+5}^d)] + \\
\pi y^d + \pi^{d-2}z^d
\end{align*}

Recall that for these four fields $2 \equiv \pi^2 \pmod{\pi^4}$ and that the $d$\textsuperscript{th} powers modulo $\pi^4$ are $0$, $1$, and $1 + \pi^2 + \pi^3$.

Suppose that $f$ has a nontrivial zero $(a_0, \ldots, a_{s-3}, b, c)$ in $\mathcal{O}$.  We have $a_0^d + a_1^d + a_2^d + \pi b^d \equiv 0 \pmod{\pi^2}$.  Since the $d$\textsuperscript{th} powers modulo $\pi^2$ are $0$ and $1$, it follows that $\pi | b$.  Note that

\begin{align*}
    (1 + \pi)(x^d + y^d + z^d) + \pi^2(u^d + v^d + w^d) \equiv \\
    (1 + \pi)[(x^d + y^d + z^d) + (\pi^2 + \pi^3)(u^d + v^d + w^d)] \pmod{\pi^4}
\end{align*}

Suppose $\pi$ divides all of $a_i$ for $i < 3j$.  Then $\pi^2$ divides $a_{3j}^d + a_{3j+1}^d + a_{3j+2}^d$.  If $\pi$ does not divide all of $a_{3j}$, $a_{3j+1}$, and $a_{3j+2}$, then modulo $\pi^4$, $a_{3j}^d + a_{3j+1}^d + a_{3j+2}^d \in \{\pi^2, \pi^3\}$, and so $a_{3j}^d + a_{3j+1}^d + a_{3j+2}^d + (\pi^2 + \pi^3)(a_{3j+3}^d + a_{3j+4}^d + a_{3j+5}^d) \not \equiv 0 \pmod{\pi^4}$.  Thus assume all of $a_{3j}$, $a_{3j+1}$, and $a_{3j+2}$ are divisible by $\pi$.  By induction, $\pi$ divides each of $a_0, \ldots, a_{s-3}$.  So, we must have $\pi | c$, a contradiction.  Thus $f$ has no primitive zero modulo $\pi^d$, and therefore $f$ has no nontrivial zero over $K$.

If $K \in \{\mathbb{Q}_2(\sqrt{-1}), \mathbb{Q}_2(\sqrt{-5})\}$, from Lemma \ref{final} we have $\Gamma(d, K) \le d+1$.  To demonstrate we have the equality from the second part of Theorem \ref{theo}, we present a form in $s = d$ variables that has no nontrivial zero over $K$.

$$g = \sum_{i=0}^{d-1} \pi^i x_i^d$$

\bibliographystyle{plain}
\bibliography{biblio}

\begin{thebibliography}{1}

\bibitem{davenport1963homogeneous}
H.~Davenport and D.~J. Lewis.
\newblock Homogeneous additive equations.
\newblock {\em Proc. Roy. Soc. London Ser. A}, 274:443--460, 1963.

\bibitem{duncan2020solubility}
Drew Duncan and David~B. Leep.
\newblock Solubility of additive sextic forms over $\mathbb{Q}_2(\sqrt {-1})$
  and $\mathbb{Q}_2(\sqrt{-5})$.
\newblock {\em arXiv preprint arXiv:2005.09770}, 2020.

\bibitem{greenberg1969lectures}
Marvin~J. Greenberg.
\newblock {\em Lectures on forms in many variables}.
\newblock W. A. Benjamin, Inc., New York-Amsterdam, 1969.

\bibitem{knapp2016solubility}
M~Knapp.
\newblock Solubility of additive sextic forms over ramified quadratic
  extensions of {$\mathbb{Q}_2$}.
\newblock {\em Publ. Math. Debrecen}, 95(1--2):67--91, 2019.

\bibitem{MR2104929}
T.~Y. Lam.
\newblock {\em Introduction to quadratic forms over fields}, volume~67 of {\em
  Graduate Studies in Mathematics}.
\newblock American Mathematical Society, Providence, RI, 2005.

\bibitem{leep2018diagonal}
David~B. Leep and Luis Sordo~Vieira.
\newblock Diagonal equations over unramified extensions of {$\Bbb Q_p$}.
\newblock {\em Bull. Lond. Math. Soc.}, 50(4):619--634, 2018.

\bibitem{MR197450}
Guy Terjanian.
\newblock Un contre-exemple \`a une conjecture d'{A}rtin.
\newblock {\em C. R. Acad. Sci. Paris S\'{e}r. A-B}, 262:A612, 1966.

\end{thebibliography}
\end{document}